\documentclass[a4paper, draft]{amsart}

\usepackage[utf8]{inputenc}
\usepackage[T1]{fontenc}
\usepackage[english]{babel}
\usepackage[style=numeric,isbn=false,url=false,eprint=false,doi=false]{biblatex}
\renewbibmacro{in:}{}
\DeclareNameAlias{sortname}{first-last}
\usepackage{cleveref}

\usepackage{booktabs}

\usepackage{amsfonts}
\usepackage{amsmath}
\usepackage{amssymb}
\usepackage{amsthm}
\usepackage{bbm}

\usepackage{xfrac}

\usepackage{thmtools}
\usepackage{thm-restate}
\usepackage[dvipsnames]{xcolor}

\usepackage{ stmaryrd } %

\usepackage{enumitem}
\setenumerate{label=(\roman*),itemsep=3pt}

\usepackage{MathMacrosRP}
\usepackage{MathMacrosThesisRP}

\declaretheorem[
name = Theorem,
]{theorem}
\declaretheorem[
name = Corollary,
sibling = theorem
]{corollary}

\declaretheorem[
name = Proposition,
sibling = theorem
]{proposition}

\declaretheorem[
name = Claim,
numbered = no
]{claim*}

\usepackage{etoolbox}
\AtEndEnvironment{proof}{\setcounter{claim}{0}}

\declaretheoremstyle[%
  spaceabove=-6pt,%
  spacebelow=6pt,%
  headfont=\normalfont\itshape,%
  postheadspace=1em,%
  qed=$\blacksquare$,%
  headpunct={.}
]{mystyle}

\declaretheorem[
name = Example,
sibling = theorem,
style=definition
]{example}
\declaretheorem[
name = Definition,
sibling = theorem,
style=definition
]{definition}

\declaretheorem[
name = Question,
sibling = theorem,
style=definition
]{question}
\declaretheorem[
name = Conjecture,
sibling = theorem,
style = definition
]{conjecture}

\usepackage{tikz}
   \usetikzlibrary{positioning,arrows.meta,calc}
   \tikzset{
   modal/.style={>=Stealth,shorten >=1pt,shorten <=1pt,auto,node distance=1.5cm,
   semithick},
   world/.style={circle,draw,minimum size=0.5cm,fill=gray!15},
   point/.style={circle,draw,inner sep=0.5mm,fill=black},
   reflexive above/.style={->,loop,looseness=7,in=120,out=60},
   reflexive below/.style={->,loop,looseness=7,in=240,out=300},
   reflexive left/.style={->,loop,looseness=7,in=150,out=210},
   reflexive right/.style={->,loop,looseness=7,in=30,out=330}
   }

\usepackage{trimclip,adjustbox}

\newcommand{\VDash}{%
  \mathrel{
    \text{\clipbox{0pt 0pt {.8\width} 0pt}{$\Vdash$}}
    \mkern.9mu
    \text{\adjustbox{width=.87\width,height=\height}{$\vDash$}}
  }
}

\renewcommand{\VDash}{\Vdash}

\begin{document}
    
\title[De Jongh's Theorem for IZF]{De Jongh's Theorem for intuitionistic Zermelo-Fraenkel set theory}
\author{Robert Passmann} %
\address{Institute for Logic, Language and Computation (ILLC), Faculty of Science, University of Amsterdam, ORCID: 0000-0002-7170-3286}
\email{robertpassmann@posteo.de}
\thanks{I would like to thank Lorenzo Galeotti, Joel Hamkins, Rosalie Iemhoff, Dick de Jongh, Yurii Khomskii, and Benedikt Löwe for helpful and inspiring discussions. This research was partially supported by the \emph{Studienstiftung des deutschen Volkes}.}
\date{\today}

\maketitle

\begin{abstract}
    We prove that the propositional logic of intuitionistic set theory $\IZF$ is intuitionistic propositional logic $\IPC$. More generally, we show that $\IZF$ has the de Jongh property with respect to every intermediate logic that is complete with respect to a class of finite trees. The same results follow for $\CZF$.
\end{abstract}

\section{Introduction}

De Jongh's classical theorem \cite{deJongh1970} states that the propositional logic of Heyting Arithmetic $\HA$ is intuitionistic logic $\IPC$. In this work, we will prove de Jongh's theorem for intuitionistic Zermelo-Fraenkel set theory $\IZF$. That is, for all propositional formulas $\phi$:

\begin{center}
    $\IPC \vdash \phi$ if and only $\IZF \vdash \phi^\sigma$ for all substitutions $\sigma$.
\end{center}

To prove this result, we introduce a new semantics for $\IZF$, the so-called \emph{blended Kripke models}, or \emph{blended models} for short. These models are inspired by the constructions of Iemhoff \cite{Iemhoff2010} and Lubarsky \cite{Lubarsky2005, Lubarsky2018, LubarskyDiener2014, LubarskyRathjen2008}, and combine Kripke semantics with classical models of set theory.

Passmann \cite{PassmannDJP} showed that Iemhoff's models for bounded constructive set theory allow to prove de Jongh's theorem for a very weak set theory, strictly weaker than $\CZF$. This is because Iemhoff's construction allows to choose a classical model of set theory at every node of the Kripke frame, thus retaining a great amount of control of the set-theoretic model on top. However, it can be shown that the same method cannot yield de Jongh's theorem for stronger set theories, such as $\CZF$ or $\IZF$.

Moreover, Passmann \cite[Chapter 4]{Passmann2018} proved that the propositional logic of those of Lubarsky's models that are based on a Kripke frame with end-nodes contains the intermediate logic $\mathbf{KC}$. Consequently, the Lubarsky models based on such a frame cannot be used to prove de Jongh properties with respect to logics weaker than $\mathbf{KC}$, such as $\IPC$. 

The blended models have more flexibility than Lubarsky's models and model a stronger set theory than Iemhoff's models, and can therefore be used to prove de Jongh's theorem for $\IZF$. 

In fact, we will prove a stronger result concerning the \emph{de Jongh property} for $\IZF$: The de Jongh property is a generalisation of de Jongh's theorem for arbitrary theories and intermediate logics (see \Cref{Definition:de Jongh property}). Our main result can be stated as follows.

\begin{theorem}[\Cref{Theorem: IZF de Jongh for finite trees}]
    Intuitionistic set theory $\IZF$ has the de Jongh property with respect to every intermediate logic $\mathbf{J}$ that is characterised by a class of finite trees. 
\end{theorem}

In particular, $\IZF$ has the de Jongh property with respect to the logics $\IPC$, $\LC$, $\mathbf{T}_n$ and $\mathbf{BD}_n$. As constructive set theory $\CZF$ is a subtheory of $\IZF$, all of these results also apply for $\CZF$ (see \Cref{Corollary: CZF has dJP with finite trees}).

The article is organised as follows. \Cref{Section: Preliminaries} discusses the necessary preliminaries for this article. We introduce blended models in \cref{Section: blended models} and prove that they satisfy intuitionistic set theory $\IZF$. In \cref{Section: Propositional logic of blended models}, we consider the propositional logic of blended models and prove de Jongh's theorem for $\IZF$. We draw some conclusions and state a few questions for further research in \cref{Section: Conclusions}.

\section{Preliminaries}
\label{Section: Preliminaries}

In this section, we will discuss the preliminaries for the later sections. After briefly discussing notation and intermediate logics in \Cref{Section: Notation} and \Cref{Section: Intuitionistic and intermediate logics}, respectively, we will introduce Kripke semantics for intuitionistic propositional logic in \Cref{Section: Kripke frames}. We will then discuss the de Jongh property in \Cref{Section: Loyalty faithfulness and dJP}. 

\subsection{Notation and meta-theory}
\label{Section: Notation}

We adopt the following notational policy: The symbol $\VDash$ will be used for the forcing relation of Kripke models. As usual, we will use $\vDash$ for the classical modelling relation, and $\vdash$ for the provability relation.

The meta-theory of this article is classical Zermelo-Fraenkel set theory $\ZFC$ assuming the existence of a countable transitive model of set theory.

\subsection{Intuitionistic and intermediate logics} 
\label{Section: Intuitionistic and intermediate logics}

We fix a countable set $\Prop$ of propositional variables for the scope of this article, and identify propositional logics $\mathbf{L}$ with the set of formulas they prove (i.e., $\mathbf{L} \vdash \phi$ if and only if $\phi \in \mathbf{L}$). As usual, we denote intuitionistic propositional logic by $\IPC$, and classical propositional logic by $\CPC$. We say that a logic $\mathbf{J}$ is an intermediate logic if $\IPC \subseteq \mathbf{J} \subseteq \CPC$ (in particular, $\IPC$ and $\CPC$ are considered intermediate logics here). Intuitionistic predicate logic is called $\IQC$.

\subsection{Kripke frames}
\label{Section: Kripke frames}

We will now introduce Kripke frames for intuitionistic logic. In particular, we will focus on Kripke frames that are trees because these will be relevant for us later.

\begin{definition}
    A \emph{Kripke frame} $(K,\leq)$ is a partial order. We call a Kripke frame $(K,\leq)$ a \emph{tree} if for every $v \in K$, the set $K^{\leq v} = \Set{w \in K}{w \leq v}$ is well-ordered by $<$, and moreover, if there is a node $r \in K$ such that $r \leq v$ for all $v \in K$ (i.e., $K$ is \emph{rooted}). A Kripke frame is called \emph{finite} whenever $K$ is finite.
\end{definition}

\begin{definition}
    Given a Kripke frame $(K,\leq)$, we define $E_K$ to be its set of end-nodes, i.e., the set of those $e \in K$ that are maximal with respect to $\leq$. A Kripke frame $(K,\leq)$ \emph{with end-nodes} is a Kripke frame such that for every $v \in K$ there is some $e \in E_K$ with $v \leq e$. Given a node $v \in K$, let $E_v$ denote the set of all end-nodes $e \in K$ such that $e \geq v$.
\end{definition}

The following combinatorial proposition will be useful later when we will determine the propositional logic of certain Kripke models. An \emph{upset $X$} in a Kripke frame $(K,\leq)$ is a set $X \subseteq K$ such that $v \in X$ and $v \leq w$ implies $w \in X$. Given a finite tree $(K,\leq)$ and a node $v \in K$, let $U_v$ be the number of upsets $X \subseteq K^{\geq v}$, where $K^{\geq v} = \Set{w \in K}{w \geq v}$.

\begin{proposition}\label{Proposition: uniquely determined nodes}
    In a finite tree $(K,\leq)$, every node $v$ is uniquely determined by $U_v$ and the set of end-nodes $e$ with $e \geq v$.
\end{proposition}
\begin{proof}
    This follows by an easy induction on trees. The base case for a tree with one node is trivial. For the construction step, where we add a new root below one or many trees, observe that the uniqueness within the branches is preserved and for the new bottom node as well, as it has strictly more up-sets $X \subseteq K^{\geq v}$ than any other node in the tree. 
\end{proof}

A \emph{Kripke model for $\IPC$} is a triple $(K,\leq,V)$ such that $(K,\leq)$ is a Kripke frame and $V: \Prop \to \P(K)$ a valuation that is persistent, i.e., if $w \in V(p)$ and $w \leq v$, then $v \in V(p)$. We can then define, by induction on propositional formulas, the forcing relation $\VDash$ for propositional logic at a node $v \in K$ in the following way:

\begin{enumerate}
    \item $K,V,v \VDash p$ if and only if $v \in V(p)$,
    \item $K,V,v \VDash \phi \wedge \psi$ if and only if $K,V,v \VDash \phi \text{ and } K,V,v \VDash \psi$,
    \item $K,V,v \VDash \phi \vee \psi$ if and only if $K,V,v \VDash \phi \text{ or } K,V,v \VDash \psi$,
    \item $K,V,v \VDash \phi \rightarrow \psi$ if and only if \\ \phantom{xxx} for all $w \geq v$, $K,V,w \VDash \phi$ implies $K,V,w \VDash \psi$,
    \item $K,V,v \VDash \bot$ holds never.
\end{enumerate}
    Sometimes we will write $v \VDash \phi$ instead of $K, V, v \VDash \phi$, and $K, V \VDash \phi$ if $K,V,v \VDash \phi$ holds for all $v \in K$. A formula $\phi$ is \emph{valid in $K$} if $K, V, v \VDash \phi$ holds for all valuations $V$ on $K$ and $v \in K$, and $\phi$ is \emph{valid} if it is valid in every Kripke frame $K$.

By induction on formulas, one can prove that the persistence of the propositional variables transfers to all formulas.

\begin{proposition}[Persistence]
    Let $(K,\leq,V)$ be a Kripke model for $\IPC$, $v \in K$ and $\phi$ be a propositional formula such that $K,v \VDash \phi$ holds. Then $K, w \VDash \phi$ holds for all $w \geq v$. \qed
\end{proposition}

We can now define the logic of a Kripke frame and of a class of Kripke frames.

\begin{definition}
    If $(K,\leq)$ is a Kripke frame for $\IPC$, we define the \emph{propositional logic $\Logic(K,\leq)$} to be the set of all propositional formulas that are valid in $K$. For a class $\mathcal{K}$ of Kripke frames, we define the \emph{propositional logic $\Logic(\mathcal{K})$} to be the set of all propositional formulas that are valid in all Kripke frames $(K,\leq)$ in $\mathcal{K}$. Given an intermediate logic $\mathbf{J}$, we say that \emph{$\mathcal{K}$ characterises $\mathbf{J}$} if $\Logic(\mathcal{K}) = \mathbf{J}$.
\end{definition}

We might write $\Logic(K)$ for $\Logic(K,\leq)$. Let us conclude this section with a few examples of logics and the classes of Kripke frames they are characterised by. For the proofs of these characterisations, we refer to the literature.

\begin{proposition}[{\cite[Theorem 6.12]{TroelstraVanDalen}}]
    \label{Proposition: IPC characterised by finite trees}
    Intuitionistic propositional logic $\IPC$ is characterised by the class of all finite trees.
\end{proposition}

\begin{example}\label{Example: Logics characterised by finite trees}
    We present some examples of logics from \cite{deJongh2011} that are characterised by classes of finite trees.
    \begin{enumerate}
        \item \emph{Dummett's logic} $\LC$ is the logic obtained by extending $\IPC$ with the scheme $(p \rightarrow q) \vee (q \rightarrow v)$. The logic $\LC$ is characterised by the class of finite linear orders.
        \item The \emph{Gabbay-de Jongh logics} $\mathbf{T}_n$, for $n \in \mathbb{N}$, are characterised by the class of finite trees which have splittings of exactly $n$, i.e., every node is either an end-node or has exactly $n$ successors. $\mathbf{T}_1$ coincides with $\LC$, and the logics $\mathbf{T}_n$ are axiomatised by the following formulas:
 $$
     \bigwedge_{k \leq n+1} \left( \left( \phi_k \rightarrow \bigwedge_{j \neq k} \phi_j \right) \rightarrow \bigwedge_{j \neq k} \phi_j \right) \rightarrow \bigwedge_{k \leq n+1} \phi_k.
 $$
        \item The \emph{logics of bounded depth $n$} $\mathbf{BD}_n$, for $n \in \mathbb{N}$, are characterised by the finite trees of depth $n$. The logic of depth $1$, $\mathbf{BD}_1$ is classical logic $\CPC$ axiomatised by $\beta_1 = ((\phi_1 \rightarrow \psi) \rightarrow \phi_1) \rightarrow \phi_1$ (Pierce's law). For every $n$, the logics $\mathbf{BD}_n$ is axiomatised by $\beta_n$, which is obtained via a recursive definition, $\beta_{n + 1} = ((\phi_{n+1} \rightarrow \beta_n) \rightarrow \phi_{n+1}) \rightarrow \phi_{n+1}$.
    \end{enumerate}
\end{example}

\subsection{The de Jongh property}
\label{Section: Loyalty faithfulness and dJP}

In this section, we will introduce the de Jongh property and provide a framework for proving it. For a detailed history of de Jongh's theorem and the de Jongh property see \cite{deJongh2011}.

\begin{definition}
    Let $\phi$ be a propositional formula and let $\sigma: \Prop \to \mathcal{L}^\mathsf{sent}_\in$ an assignment of propositional variables to $\mathcal{L}_\in$-sentences. By $\phi^\sigma$ we denote the $\mathcal{L}_\in$-sentence obtained from $\phi$ by replacing each propositional variable $p$ with the sentence $\sigma(p)$.
\end{definition}

The de Jongh property is a generalisation of de Jongh's classical result \cite{deJongh1970} concerning Heyting arithmetic $\mathsf{HA}$ and intuitionistic propositional logic $\mathbf{IPC}$. 

\begin{theorem}[de Jongh, {\cite{deJongh1970}}]\label{Theorem:de Jongh}
    Let $\phi$ be a formula of propositional logic. Then $\mathsf{HA} \vdash \phi^\sigma$ for all $\sigma: \Prop \to \mathcal{L}_\mathsf{HA}^\mathsf{sent}$ if and only if $\IPC \vdash \phi$. 
\end{theorem}

Given a theory based on intuitionistic logic, we may consider its propositional logic, i.e., the set of propositional formulas that are derivable after substituting the propositional letters by arbitrary sentences in the language of the theory.

\begin{definition}
    Let $\mathsf{T}$ be a theory in intuitionistic predicate logic, formulated in a language $\mathcal{L}$. A propositional formula $\phi$ will be called \emph{$\mathsf{T}$-valid} if and only if $\mathsf{T} \vdash \phi^\sigma$ for all $\sigma: \Prop \to \mathcal{L}^\mathsf{sent}$. The \emph{propositional logic $\Logic(\mathsf{T})$} is the set of all $\mathsf{T}$-valid formulas.
\end{definition}

Given a theory $\mathsf{T}$ and an intermediate logic $\mathbf{J}$, we denote by $\mathsf{T}(\mathbf{J})$ the theory obtained by closing $\mathsf{T}$ under $\mathbf{J}$.

\begin{definition}\label{Definition:de Jongh property}
    We say that a theory $\mathsf{T}$ has the \emph{de Jongh property} if $\Logic(\mathsf{T}) = \IPC$. The theory $\mathsf{T}$ has the \emph{de Jongh property with respect to an intermediate logic $\mathbf{J}$} if $\Logic(\mathsf{T}(\mathbf{J})) = \mathbf{J}$.
\end{definition}

De Jongh's theorem is equivalent to the assertion that Heyting arithmetic has the de Jongh property. 

\section{Blended models}
\label{Section: blended models}

In this section, we will construct the blended models and show that they are models of intuitionistic Zermelo-Fraenkel set theory $\IZF$.  

\subsection{Constructing blended models}
We will now construct a blended model based on a given Kripke frame $(K,\leq)$ with end-nodes. To every end-node $e \in K$, we associate a transitive transitive model $M_e$ of set theory such that $\Ord^{M_{e_0}} = \Ord^{M_{e_1}}$ for all $e_0, e_1 \in E_K$. For the remainder of this section, fix such a collection $\Seq{M_e}{e \in E_K}$. Note that $\Ord^{M_e}$ denotes the same ordinal in the meta-universe for all $e \in E_K$; we can therefore refer to this ordinal by $\Ord^{M_e}$ without specifying a particular $e \in E_K$.

The construction of the blended models happens in three steps. We begin by constructing the collection of domains $\Seq{\mathcal{D}_v}{v \in K}$: First the domains for the end-nodes and, secondly, for all remaining nodes of the Kripke frame. The third step is to define the semantics of the $\in$-relation.

\begin{description}
\item[Step 1. Domains for end-nodes] Given an end-node $e \in E_K$ of $(K,\leq)$, and the associated model $M_e$, we define a function $f_e: M_e \to V$ as follows by $\in$-recursion:
\begin{equation*}
    f_e(x) = (e, f_e[x]).
\end{equation*}
Let $\mathcal{D}_e = f_e[M_e]$. Hence, each $\mathcal{D}_e$ is a set of functions $x: K^{\geq e} \to \ran(x)$ (where $K^{\geq e} = \set{e}$). Moreover, for $\alpha \in \Ord^M$, let $\mathcal{D}_e^\alpha = f_e[(V_\alpha)^{M_e}]$. Then $\mathcal{D}_e^0 = \emptyset$ and it holds that $$\bigcup_{\alpha \in \Ord^M} \mathcal{D}_e^\alpha = \mathcal{D}_e.$$

In \Cref{Proposition: end-nodes isomorphic to original models} below, we will see that the domains of the end-nodes of a blended model are isomorphic (with respect to the equality and membership relations) to the classical model of set theory associated to the node.

\item[Step 2. Domains for all nodes]
Now we are ready to define the domains at the remaining nodes. We do this simultaneously for all $v \in K \setminus E_K$ by induction on $\alpha \in \Ord^{M_e}$. Let $\mathcal{D}_v^\alpha$ consist of the functions $x: K^{\geq v} \to \ran(x)$ such that the following properties hold:
\begin{enumerate}
    \item for all end-nodes $e \geq v$, we have $x \upharpoonright \set{e} \in \mathcal{D}_e^\alpha$,
    \item for all non-end-nodes $w \geq v$, we have $x(w) \subseteq \bigcup_{\beta < \alpha} \mathcal{D}_w^\beta$, and
    \item for all nodes $u \geq w \geq v$ we have that $\Set{y \upharpoonright K^{\geq u}}{y \in x(w)} \subseteq x(u)$.
\end{enumerate}
We use the restriction maps $f_{wu}$ with $x \mapsto x \upharpoonright K^{\geq u}$ as transition functions for the domains. Note that this map is well-defined by the definition of the domains. In particular, $f_{wu} \circ f_{vw} = f_{vu}$. 

By the definition of the maps $f_{vw}$, it is clear that condition (i) is just a special case of condition (iii). We state it separately as it requires special attention when working with blended models.

Finally, we define the domain $\mathcal{D}_v$ at the node $v$ to be the set $$\mathcal{D}_v = \bigcup_{\alpha \in \Ord^M} \mathcal{D}_v^\alpha.$$ 

\item[Step 3. Defining the semantics] 
We define, by induction on $\mathcal{L}_\in$-formulas, the forcing relation at every node of the Kripke model in the following way, where $\phi$ and $\psi$ are formulas with all free variables shown, and, moreover, $\bar y = y_0,\dots,y_{n-1}$ are elements of $\mathcal{D}_v$ for the node $v$ considered on the left side: 
    \begin{enumerate}
        \item $(K,\leq,\mathcal{D}),v \VDash x \in y$ if and only if $x \in y(v)$,
        \item $(K,\leq,\mathcal{D}),v \VDash a = b$ if and only if $a = b$,
        \item $(K,\leq,\mathcal{D}),v \VDash \exists x \, \phi(x,\bar y)$ if and only if there is some $a \in D_v$ \\ with $(K,\leq,\mathcal{D}),v \VDash \phi(a, \bar y)$,
        \item $(K,\leq,\mathcal{D}),v \VDash \forall x \, \phi(x,\bar y)$ if and only if for all $w \geq v$ and $a \in D_w$ \\ we have $(K,\leq,\mathcal{D}),w \VDash \phi(a,\bar y)$.
    \end{enumerate}
    The cases for $\rightarrow$, $\wedge$, $\vee$ and $\bot$ are analogous to the ones in the above definition of the forcing relation for Kripke models for $\IPC$. 
\end{description}

This finishes the definition of the blended models. 

\begin{definition}
    We call $(K,\leq,\mathcal{D})$ the \emph{blended Kripke model obtained from $\Seq{M_e}{e \in E_K}$}. 
\end{definition}

If the collection $\Seq{M_e}{e \in E_K}$ is either clear from the context, or if it does not matter, we will also say that $(K,\leq,\mathcal{D})$ a \emph{blended Kripke model}. We will usually say \emph{blended model} instead of blended Kripke model.

An $\mathcal{L}_\in$-formula $\phi$ is \emph{valid in $(K,\leq,\mathcal{D})$} if $v \VDash \phi$ holds for all $v \in K$, and $\phi$ is \emph{valid} if it is valid in every Kripke frame $K$. We will call $(K,\leq)$ the \emph{underlying Kripke frame} of $(K,\leq,\mathcal{D})$, or the frame that $(K,\leq,\mathcal{D})$ is based on. Moreover, let us call $\heyting{\phi}^{(K,\leq,D,e)} = \Set{v \in K}{v \VDash \phi}$ the \emph{truth set} of a sentence $\phi$ in the language of set theory in a blended model $(K,\leq,D,e)$. When the model is clear from the context, we will also write $\heyting{\phi}^K$ or just $\heyting{\phi}$.

Before we continue with some basic properties of the blended models, let us briefly discuss this construction in comparison to Lubarsky's Kripke models \cite{Lubarsky2005, Lubarsky2018, LubarskyDiener2014, LubarskyRathjen2008}, which are constructed in a similar way. The crucial difference, however, is that our models are constructed in a top-down manner that allows to choose classical models of set theory at the end-nodes, whereas Lubarsky's bottom-up construction requires elementary equivalence.

\subsection{Basic properties}

We will now observe some basic properties of the blended models.

\begin{proposition}[Persistence]
    Let $(K,\leq,\mathcal{D})$ be a blended model and $\phi$ a formula in the language of set theory. If $v \VDash \phi(a_0,\dots,a_{n-1})$ and $w \geq v$, then $w \VDash \phi(f_{vw}(a_0),\dots,f_{vw}(a_{n-1}))$. 
\end{proposition}
\begin{proof}
    This is proved by induction on $\mathcal{L}_\in$-formulas
\end{proof}

\begin{proposition}
     The blended models are sound with respect to $\IQC$.
\end{proposition}
\begin{proof}
    This follows from the more general soundness result for Kripke models for predicate logics with respect to $\IQC$. See, for example, \cite[Theorem 6.6]{TroelstraVanDalen}.
\end{proof}

We will now make the essential observation that the domains at the end-nodes are isomorphic to the models they were obtained from.

\begin{proposition}\label{Proposition: end-nodes isomorphic to original models}
    Let $(K,\leq,\mathcal{D})$ be a blended model, and $e \in E_K$. Then $(K,\leq,\mathcal{D}), e \VDash \phi(f_e(a_0),\dots,f_e(a_{n-1}))$ if and only if $M_e \vDash \phi(a_0,\dots,a_{n-1})$ for all elements $a_0,\dots,a_{n-1} \in M_e$.
\end{proposition}
\begin{proof}
    Let us first argue that $f_e: M_e \to \mathcal{D}_e$ is a bijection. Define $g$ by $\in$-recursion with $(e, x) \mapsto g[x]$. It follows by induction that $g \circ f_e = \id_{M_e}$ and $f_e \circ g = \id_{\mathcal{D}_e}$. Hence, $f_e$ is a bijection.
    
    It suffices to prove the claim for the atomic cases: equality and set-membership. The case for equality follows from the definition of the semantics and the fact that $f$ is bijective. For set-membership observe that if $M_e \vDash x \in y$, then $f_e(x) \in f_e(y)(e)$ and hence $e \VDash f_e(x) \in f_e(y)$. Conversely, if $e \VDash f_e(x) \in f_e(y)$, then $f_e(x) \in f_e(y)(e)$ and hence $x = g(f_e(x)) \in g(f_e(y)) = y$. The other cases follow trivially as the intuitionistic interpretation of the logical symbols in an end-node coincides with the classical interpretation in the model $M_e$.
\end{proof}

\subsection{$\IZF$ holds in blended models}

In this section, we will show that the axioms of $\IZF$ (see \Cref{Figure: IZF}) hold in blended models. For the sake of this section, let $(K,\leq,\mathcal{D})$ be a blended model obtained from $\Seq{M_e}{e \in E_K}$.

Note that $\IZF$ trivially holds true at every end-node because the models associated with the end-nodes are, in fact, models of $\ZF$ set theory.

\begin{figure}
    \begin{description}
        \item[Extensionality] $\forall a \forall b (\forall x(x \in a \leftrightarrow x \in b) \rightarrow a = b)$
        \item[Empty set] $\exists a \ \forall x \in a \ \bot$
        \item[Pairing] $\forall a \forall b \exists y \forall x (x \in y \leftrightarrow (x = a \vee x = b))$
        \item[Union] $\forall a \exists y \forall x (x \in y \leftrightarrow \exists u (u \in a \wedge x \in u))$
        \item[Power set] $\forall a \exists y \forall x (x \in y \leftrightarrow  x \subseteq a)$
        \item[Infinity]
        $ \exists a (\exists x \ x \in a \wedge \forall x \in a \exists y \in a \ x \in y) $
        \item[Set Induction] $(\forall a (\forall x \in a \ \phi(x) \rightarrow \phi(a))) \rightarrow \forall a \phi(a),$ for all formulas $\phi(x)$.
        \item[Separation]
        $ \forall a \exists y \forall x (x \in y \leftrightarrow (x \in a \wedge \phi(x))),$
        for all formulas $\phi(x)$.
        \item[Collection]
        $\forall a (\forall x \in a \exists y \ \phi(x,y) \rightarrow \exists b \forall x \in a \exists y \in b \ \phi(x,y)),$
        for all formulas $\phi(x,y)$, where $b$ is not free in $\phi(x,y)$.
    \end{description}
    
    \caption{The axioms of $\IZF$.}
    \label{Figure: IZF}
\end{figure}

\begin{proposition}
    The model $(K,\leq,\mathcal{D})$ satisfies the axiom of extensionality.
\end{proposition}
\begin{proof}
    Let $v \in K$ and $a, b \in \mathcal{D}_v$. We have to show that $$v \VDash \forall x(x \in a \leftrightarrow x \in b) \rightarrow a = b.$$
    So assume that $w \VDash \forall x(x \in a \leftrightarrow x \in b)$ for all $w \geq v$, i.e., $a(w) = b(w)$ for all $w \geq v$. Hence, $a$ and $b$ are equal as functions with domain $K^{\geq v}$, and so they are equal.
\end{proof}

\begin{proposition}\label{Prop: axiom of pairing}
    The model $(K,\leq,\mathcal{D})$ satisfies the axiom of pairing.
\end{proposition}
\begin{proof}
    Let $v \in K$ and $a, b \in \mathcal{D}_v$. Let $c$ be the function with $c(w) = \set{f_{vw}(a), f_{vw}(b)}$ for all $w \geq v$. 
    
    Let us first show that $c \in \mathcal{D}_v$. For condition (i), let $e \geq v$ be an end-node. As $a, b \in \mathcal{D}_v$ it follows from the definition that $f_{ve}(a), f_{ve}(b) \in \mathcal{D}_e$. Hence, by pairing in $M_e$, we have that $c \upharpoonright \set{e} \in \mathcal{D}_e$, where $c(e) = \set{f_{ve}(a), f_{ve}(b)}$. Conditions (ii) and (iii) of the definition of $\mathcal{D}_v$ follow directly from the definition of $c$.
    
    Now it is straightforward to check that $c$ constitutes a witness for the axiom of pairing for $a$ and $b$ at the node $v$. 
\end{proof}

\begin{proposition}
    The model $(K,\leq,\mathcal{D})$ satisfies the axiom of union.
\end{proposition}
\begin{proof}
    Let $v \in K$ and $a \in \mathcal{D}_v$. Define a function $b$ with domain $K^{\geq v}$ with $b(w) = \bigcup_{c \in a(w)} c(w)$ for all $w \geq v$. 
    
    Again, we need to show that $b \in \mathcal{D}_v$. For condition (i), observe that $f_{ve}(a) \in \mathcal{D}_e$ for every end-node $e \geq v$. As the axiom of union holds in $M_e$, it follows that there is a witness $b' \in \mathcal{D}_e$. By transitivity of $M_e$, it must then hold that $b \upharpoonright \set{e} = b' \in \mathcal{D}_e$. As in the previous proposition, conditions (ii) and (iii) follow directly from the definition of $b$. Then $b$ witnesses the axiom of union for $a$.
\end{proof}

\begin{proposition}\label{Prop: axiom of empty set}
    The model $(K,\leq,\mathcal{D})$ satisfies the axiom of empty set.
\end{proposition}
\begin{proof}
    For every $v \in K$ consider the empty function with domain $K^{\geq v}$. This is an element of $\mathcal{D}_v$ and witnesses the axiom of empty set.
\end{proof}

\begin{proposition}
    The model $(K,\leq,\mathcal{D})$ satisfies the axiom of infinity.
\end{proposition}
\begin{proof}
    By recursion on natural numbers, we will define elements $n_v \in \mathcal{D}_v$ simultaneously for every $v \in K$. Let $0_v$ be the empty set as defined in the proof of \Cref{Prop: axiom of empty set}. Then, if $m_v$ has been defined for all $m < n$, let $n_v$ be the function with $n_v(w) = \set{0_w,\dots,(n-1)_w}$ for all $w \geq v$. This finishes the recursive definition. It follows inductively that every $n_v \in \mathcal{D}_v$, again paying special attention at the end-nodes: the sets $n_e$ correspond to the finite ordinal $n \in M_e$.
    
    Finally, let $\omega_v(w) = \Set{n_w}{n < \omega}$ for all $w \geq v$. To see that $\omega_v \in \mathcal{D}_v$ note that, for every end-node $e \geq v$, $f_{ve}(\omega_v) = \omega_e \in \mathcal{D}_e$ as $M_e$ satisfies the axiom of infinity.
    
    It follows that $\omega_v$ is a witness for the axiom of infinity at the node $v$. 
\end{proof}

\begin{proposition}
    The model $(K,\leq,\mathcal{D})$ satisfies the axiom scheme of separation.
\end{proposition}
\begin{proof}
    Let $\phi(x,y_0,\dots,y_n)$ be a formula with all free variables shown. Let $v \in K$, $a \in \mathcal{D}_v$ and $b_0, \dots, b_n \in \mathcal{D}_v$. Define $c$ to be the function with domain $K^{\geq v}$ such that 
    $
        c(w) = \Set{d \in a(w)}{ w \VDash \phi(d,b_0,\dots,b_n)}
    $
    holds for all $w \geq v$. We have that $c \in \mathcal{D}_v$ by the definition of the domains $\mathcal{D}_v$. Again, property (i) follows from the fact that separation holds in $M_e$ for every end-node model $M_e$. Moreover, property (iii) follows by persistence. Finally, $c$ witnesses separation from $a$ by $\phi$ with parameters $b_i$. 
\end{proof}

\begin{proposition}
    If $K$ is finite, then the model $(K,\leq,\mathcal{D})$ satisfies the axiom scheme of collection.
\end{proposition}
\begin{proof}
    Let $v \in K$, $\phi(x,y)$ be a formula (possibly with parameters), and $a \in \mathcal{D}_v$. We need to show that:
    \begin{align*}
        v \VDash \forall x \in a \exists y \ \phi(x,y) \rightarrow \exists b \forall x \in a \exists y \in b \ \phi(x,y).
    \end{align*}
    So let us additionally assume that $v \VDash \forall x \in a \exists y \phi(x,y)$, i.e., for every $w \geq v$ and every $x \in a(w)$ there exists some $y \in \mathcal{D}_w$ such that $w \VDash \phi(x,y)$. Let $\alpha$ be the minimal ordinal such that for every $w \geq v$ and $x \in a(w)$, there is some $y \in \mathcal{D}_w^\alpha$ with $w \VDash \phi(x,y)$. Note that $\alpha < \Ord^{M_e}$ as $K$ is finite. Define $b$ to be the function with domain $K^{\geq v}$ such that $b(v) = \mathcal{D}_w^\alpha$. It follows that $b \in \mathcal{D}_v$, where the case for end-nodes $e$ follows from the fact that $(V_\alpha)^{M_e}$ is a set in $M_e$. Hence, $b$ is a witness for the above instance of the collection scheme.
\end{proof}

\begin{proposition}
    The model $(K,\leq,\mathcal{D})$ satisfies the powerset axiom.
\end{proposition}
\begin{proof}
    Let $v \in K$ and $a \in \mathcal{D}_v$. Define a function $b$ with domain $K^{\geq v}$ such that 
    $$
        b(w) = \Set{c \in \mathcal{D}_w}{w \VDash c \subseteq f_{vw}(a)}
    $$
    for all $w \geq v$. We have to show that $b \in \mathcal{D}_v$. Observe that for every end-node $e \geq v$, $f_{ve}(b)$ corresponds to $(\P(a))^{M_e}$, and hence condition (i) is satisfied. Conditions (ii) and (iii) follow easily.
\end{proof}

\begin{proposition}
    The model $(K,\leq,\mathcal{D})$ satisfies the axiom scheme of set induction.
\end{proposition}
\begin{proof}
    The scheme of set induction follows by an induction on the rank $\alpha$ of $\mathcal{D}_v^\alpha$ simultaneously for all $v \in K$. In particular, we use the fact that the models $M_e$ associated with the end-nodes are transitive models of set theory.
\end{proof}

Let us summarise the results of this section in the following theorem. 

\begin{theorem}
    If $K$ is finite, then the model $(K,\leq,\mathcal{D})$ satisfies $\IZF$. For arbitrary $K$, the model $(K,\leq,\mathcal{D})$ satisfies $\IZF - \mathrm{Collection}$. \qed
\end{theorem}

\subsection{An example} 

Let us consider an easy example of a blended model with a failure of the law of excluded middle.

\begin{example}
    To illustrate our construction above, we will construct a Kripke model $(K,\leq,\mathcal{D})$ such that $(K,\leq,\mathcal{D}) \not \VDash \mathsf{CH} \vee \neg\mathsf{CH}$. Take $(K,\leq)$ to be the three element Kripke frame $(K,\leq)$ with $K = \set{v,e_0,e_1}$ with $\leq$ being the reflexive closure of the relation defined by $v \leq e_0$ and $v \leq e_1$ (see \Cref{Figure: blended model with CH and not CH}).
    
    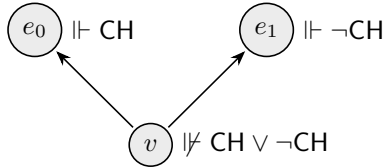
\begin{figure}[h]
    \begin{tikzpicture}[modal]
      \node[world] (v) [label=right:$\not \VDash \mathsf{CH} \vee \neg\mathsf{CH}$] {$v$};
      \node[world] (e_0) [above left=of v, label=right:$\VDash \mathsf{CH}$] {$e_0$};
      \node[world] (e_1) [above right=of v, label=right:$\VDash \neg\mathsf{CH}$] {$e_1$};
      
      \path[->] (v) edge (e_0);
      \path[->] (v) edge (e_1);
    \end{tikzpicture}
    
    \caption{An example of a blended model constructed from end-models $M_{e_0}$ and $M_{e_1}$ with $M_{e_0} \vDash \mathsf{CH}$ and $M_{e_0} \vDash \neg \mathsf{CH}$.}
    \label{Figure: blended model with CH and not CH}
    \end{figure}
    
    Now, let $M$ be any countable transitive model of $\ZFC + \CH$, and take $G$ to be generic for Cohen forcing over $M$. Then we associate the model $M$ with node $e_0$, and $M[G]$ with $e_1$, i.e., $M_{e_0} = M$ and $M_{e_1} = M[G]$. By our construction above and \Cref{Proposition: end-nodes isomorphic to original models}, $(K,\leq,\mathcal{D}), e_0 \VDash \mathsf{CH}$ and $(K,\leq,\mathcal{D}), e_1 \VDash \neg\mathsf{CH}$. Hence, by persistence, $(K,\leq,\mathcal{D}), v \not \VDash \mathsf{CH} \vee \neg\mathsf{CH}$.
    
    In particular, this also shows that $\IZF \not \vdash \mathsf{CH} \vee \neg\mathsf{CH}$, i.e., the law of excluded middle does not hold for assertions regarding the continuum.
\end{example}

One can easily generalise the above argument to obtain the following proposition.

\begin{proposition}
    If $\phi$ is a sentence in the language of set theory such that there are models $M$ and $N$ of $\ZFC$ with $M \vDash \phi$ and $N \vDash \neg \phi$, then $\IZF \not \vdash \phi \vee \neg \phi$. \qed
\end{proposition}

\section{The propositional logic of blended models}
\label{Section: Propositional logic of blended models}

In this section, we will analysis the propositional logic of blended models and prove the de Jongh property for $\IZF$ with respect to intermediate logics that are characterised by a class of finite trees.

\subsection{Faithful blended models.} 
\label{Section: Faithful classes of blended models}
The aim of this section is to show that we can find a blended model based on every finite tree Kripke frame $(K,\leq)$ that allows us to imitate every valuation on $(K,\leq)$. Let us begin with a definition and several useful observations.

\begin{definition}
    An blended model $(K,\leq,\mathcal{D})$ is called \emph{faithful} if for every valuation $V$ on the Kripke frame $(K,\leq)$ and every propositional letter $p \in \Prop$, there is an $\mathcal{L}_\in$-formula $\phi_p$ such that $\heyting{\phi_p}^{(K,\leq,\mathcal{D})} = V(p)$.
\end{definition}

This notion was first introduced in \cite{Passmann2018}. For further discussion and connections to the de Jongh property, see also \cite{PassmannDJP}.

Given a natural number $n$, let $\psi_n$ be the following sentence\footnote{These sentences were discovered in a discussion with Lorenzo Galeotti and Benedikt Löwe regarding the logics of algebra-valued models of set theory, see also \cite{LoewePassmannTarafder2018}. We adapt them here for the case of Kripke semantics.} in the language of set theory:

\begin{equation*}
    \forall x_0, \dots, x_{n-1} \left( \bigwedge_{i < n} (\forall y \in x_i \forall z \in y \, \bot) \rightarrow  \bigwedge_{i < j < n} x_i = x_j \right).
\end{equation*}
Informally, this sentence asserts that given $n$ subsets of $1$, at least $2$ of them are equal.

The following proposition holds for all Kripke frames with end-nodes and not only for finite trees. We also do not need to assume that $U_v$ is finite; recall that we defined $U_v$ in \Cref{Section: Kripke frames} to be the number of upsets $X \subseteq K^{\geq v}$. 

\begin{proposition}\label{Proposition: counting upsets}
    Let $(K,\leq)$ be a Kripke frame with end-nodes, $(K,\leq,\mathcal{D})$ be a blended model, and $v \in K$. For every natural number $n$, we have that $v \VDash \psi_{n + 1}$ if and only if $n \geq U_v$.
\end{proposition}
\begin{proof}
    Given any upset $X \subseteq K^{\geq v}$, we define the element $1^v_X$ to be the function 
    $$K^{\geq v} \to \bigcup_{w \geq v} \mathcal{D}_w, \ \
    w \mapsto \begin{cases}
        \set{0_w}, &\text{if } w \in X, \\
        \emptyset, &\text{otherwise.}
    \end{cases}$$
    Observe that $1^v_X \in \mathcal{D}_v$ as it is monotone because $X$ is an upset. Further, we have $1^v_X \neq 1^v_Y$ for upsets $X \neq Y$ and therefore, $v \not \VDash 1^v_X = 1^v_Y$. It follows that $v \VDash \forall y \in 1^v_X \forall z \in y \, \bot$ for all upsets $X$ because $1^v_X(w)$ is either empty or contains the empty set for $w \geq v$. We conclude that $v \not \VDash \psi_{n+1}$ for $n < U_v$ taking the $1^v_X$ as witnesses.
    
    Conversely, assume that $n \geq U_v$. We will first show that whenever $v \VDash \forall y \in x \forall z \in y \, \bot$ for some $x \in \mathcal{D}_v$, then $x$ is actually of the form $1^v_X$ for some upset $X \subseteq K^{\geq v}$. For contradiction, assume that $x$ was not of the form $1^v_X$ for some upset $X$. Then there is a node $w \geq v$ such that $x(w)$ contains an element $y$ different from $0_w$. But then there must be a node $u \geq w$ such that $y(w)$ is non-empty. This is a contradiction to $v \VDash \forall y \in x \forall z \in y \, \bot$, and hence, every element $x \in \mathcal{D}_v$ satisfying the above formula must be of the form $1^v_X$. As there are only $U_v$-many elements $1^v_X$, we know that the conclusion of $\psi_{n+1}$ must be true at the node $v$. Hence, $v \VDash \psi_{n+1}$.
\end{proof}

The following proposition is a special case of a more general proposition for Kripke models of predicate logic. 

\begin{proposition}\label{Proposition: blended model endnodes labelling}
    Let $(K,\leq)$ be a Kripke frame with end-nodes, $(K,\leq,\mathcal{D})$ be a blended model and $v \in K$. If $e \not \VDash \phi$ for all end-nodes $e \geq v$, then $v \VDash \neg \phi$.
\end{proposition}
\begin{proof}
    By the definition of our semantics, we know that $v \VDash \neg \phi$ if and only if $w \not \VDash \phi$ for all $w \geq v$. Assume that there was a node $w \geq v$ such that $w \VDash \phi$. By persistence we can conclude that $e \VDash \phi$ for every end-node $e \geq w$, a contradiction.
\end{proof}

\begin{theorem}\label{Theorem: blended model with separated end-nodes is faithful}
    Let $(K,\leq,\mathcal{D})$ be a blended model based on a finite tree $(K,\leq)$ with end-nodes $e_0,\dots,e_{n-1}$. If there is a collection of $\in$-sentences $\phi_i$ for $i < n$ such that $e_j \VDash \phi_i$ if and only if $i = j$, then $(K,\leq,\mathcal{D})$ is faithful.
\end{theorem}
\begin{proof}
    Let $(K,\leq,\mathcal{D})$ be a blended model based on a finite tree $(K,\leq)$ with end-nodes $e_0,\dots,e_{n-1}$ such that there is a collection of $\in$-sentences $\phi_i$ for $i < n$ such that $e_j \VDash \phi_i$ if and only if $i = j$.
    
    As $(K,\leq)$ is a finite tree, we know by \Cref{Proposition: uniquely determined nodes} that every node $v \in K$ is uniquely determined by $U_v$ and the set of end-nodes $e \geq v$. 
    
    Let $V$ be a valuation on $(K,\leq)$. For every $p \in \Prop$, we need to find a sentence $\psi_p$ in the language of set theory such that $\heyting{\psi_p}^{(K,\leq,\mathcal{D})} = V(p)$. Due to finiteness of $K$, it suffices to consider upsets of the form $K^{\geq v}$ for some $v \in K$ because more general upsets can be covered by disjunctions.
    
     We will now prove for every $v \in K$ that there is a sentence $\chi_v$ in the language of set theory such that $(K,\leq,\mathcal{D}), w \VDash \chi_v$ if and only if $w \geq v$ (i.e., $w \in K^{\geq v}$). Let $\chi_v$ be the following sentence, where $n = U_v + 1$:
    $$
        \psi_{n} \wedge \bigwedge_{e_i \not\geq v} \neg \phi_i 
    $$
    By \Cref{Proposition: counting upsets} and \Cref{Proposition: blended model endnodes labelling} it is clear that $w \VDash \chi_v$ for all $w \geq v$. For the converse direction, let $w \in K$ such that $w \not \geq v$. There are two cases. 
    
    First, if $w < v$, then $U_w > U_v = n$ and hence $w \not \VDash \psi_n$ by \Cref{Proposition: counting upsets}. Hence, it follows that $w \not \VDash \chi_v$.
    
    Second, if $w \not < v$, then there must be an end-node $e_i \geq w$ such that $e_i \not \geq v$. By assumption $e_i \VDash \phi_i$ and hence, $w \not \VDash \neg \phi_i$. But this means that $w \not \VDash \chi_v$.
    
    This concludes the proof of the theorem.
\end{proof}

\begin{theorem}\label{Theorem: Faithful model based on finite tree}
    Let $(K,\leq)$ be a finite tree. Then there is a faithful blended model $(K,\leq,\mathcal{D})$ based on $(K,\leq)$.
\end{theorem}
\begin{proof}
    Let $e_0,\dots,e_{n-1}$ be the set of end-nodes of $(K,\leq)$. Let $M$ be a countable transitive model of $\ZFC$ set theory. By set-theoretic forcing, we can obtain generic extensions $M[G_i]$ of $M$ such that $M[G_i] \vDash 2^{\aleph_0} = \aleph_{i+1}$ for every $i < n$ (see, e.g., \cite{Kunen1980} for details). Let $M_{e_i} = M[G_i]$, and $(K,\leq,\mathcal{D})$ be the blended model obtained from $\Seq{M_i}{i < n}$. Clearly,  $M_{e_i} \vDash 2^{\aleph_0} = \aleph_{j+1}$ if and only if $i = j$. This implies, by \Cref{Proposition: end-nodes isomorphic to original models}, that $e_i \VDash 2^{\aleph_0} = \aleph_{j+1}$ if and only if $i = j$. In this situation, we can apply \Cref{Theorem: blended model with separated end-nodes is faithful} to conclude that $(K,\leq,\mathcal{D})$ is faithful.
\end{proof}

\subsection{The de Jongh property for $\IZF$}

In this section, we will draw conclusions regarding the de Jongh property of $\IZF$ from the main result of the previous section.

\begin{theorem}\label{Theorem: IZF de Jongh for finite trees}
    Intuitionistic set theory $\IZF$ has the de Jongh property with respect to every intermediate logic $\mathbf{J}$ that is characterised by a class of finite trees.
\end{theorem}
\begin{proof}
    Let $\mathbf{J}$ be an intermediate logic with $\Logic(\mathcal{K}) = \mathbf{J}$, where $\mathcal{K}$ is a class of finite trees. We have to show that $\Logic(\IZF(\mathbf{J})) = \mathbf{J}$, i.e., for every propositional formula, we have that:
    \begin{equation*}
        \mathbf{J} \vdash \phi \text{ if and only if } \IZF(\mathbf{J}) \vdash \phi^\sigma \text{ for all substitutions } \sigma: \Prop \to \mathcal{L}_\in^\sent.
    \end{equation*}
    The direction from left to right is immediate from the definition of $\IZF(\mathbf{J})$. We will prove the converse direction by contraposition.
    
    Assume that there is $\phi$ such that $\mathbf{J} \not \vdash \phi$. As $\mathbf{J}$ is characterised by $\mathcal{K}$, there is a frame $(K,\leq) \in \mathcal{K}$ and a valuation $V$ such that $(K,\leq), V \not \Vdash \phi$. By \Cref{Theorem: Faithful model based on finite tree} and the assumption that $\mathcal{K}$ consists of finite trees, we can find a faithful blended model $(K,\leq,\mathcal{D})$ based on $(K,\leq)$. For every propositional letter $p \in \Prop$, let $\psi_p$ be a sentence in the language of set theory such that $\heyting{\psi_p}^{(K,\leq,\mathcal{D})} = V(p)$. Define an assignment $\sigma: \Prop \to \mathcal{L}_\in^\sent$ by $\sigma(p) = \psi_p$. 
    
    We prove by induction on propositional formulas $\chi$, simultaneously for all $v \in K$ that:
    $$
        (K,\leq), v \Vdash \chi \text{ if and only if } (K,\leq,\mathcal{D}),v \VDash \chi^\sigma.
    $$
    
    The base case for propositional letters follows directly from the definition of $\sigma$. Furthermore, the induction cases for the connectives $\rightarrow$, $\wedge$ and $\vee$ follow directly from the fact that their semantics coincide in Kripke models for $\IPC$ and in blended models. This finishes the induction.
    
    Hence, it follows from the induction that $(K,\leq,\mathcal{D}) \not \VDash \phi^\sigma$, and therefore, $\phi \notin \Logic(\IZF(\mathbf{J}))$. This finishes the proof of the theorem.
\end{proof}

\begin{corollary}
    Intuitionistic set theory $\IZF$ has the de Jongh property.
\end{corollary}
\begin{proof}
    By \Cref{Proposition: IPC characterised by finite trees}, we know that $\IPC$ is complete with respect to the class of all finite trees, i.e., this class characterises $\IPC$. By the previous \Cref{Theorem: IZF de Jongh for finite trees}, this implies that $\IZF$ has the de Jongh property.
\end{proof}

More examples of logics that are characterised by classes of finite trees are Gödel-Dummett logic $\LC$, the Gabbay-de Jongh logics $\mathbf{T}_n$, and the logics of bounded depth $\mathbf{BD}_n$ (see \Cref{Example: Logics characterised by finite trees}).

\begin{corollary}
    Intuitionistic set theory $\IZF$ has the de Jongh property with respect to the logics $\LC$, $\mathbf{T}_n$ and $\mathbf{BD}_n$. \qed
\end{corollary}

In \cite{PassmannDJP}, the author proved that bounded constructive set theory has the de Jongh property with respect to every Kripke-complete logic. With the techniques used there, it was (provably) not possible to extend the result to $\CZF$. However, we can now derive a result for $\CZF$.

\begin{corollary}\label{Corollary: CZF has dJP with finite trees}
    Constructive set theory $\CZF$ has the de Jongh property with respect to every intermediate logic $\mathbf{J}$ that is characterised by a class of finite trees.
\end{corollary}
\begin{proof}
    We have to show that $\mathbf{J} \vdash \phi$ if and only if $\CZF(\mathbf{J}) \vdash \phi^\sigma$ for all $\sigma: \Prop \to \mathcal{L}_\in^\sent$. The implication from left to right is trivial. We prove the other direction by contraposition. So assume that $\mathbf{J} \not \vdash \phi$. By \Cref{Theorem: IZF de Jongh for finite trees}, there is some $\sigma$ such that $\IZF(\mathbf{J}) \not \vdash \phi^\sigma$. As $\CZF \subseteq \IZF$, it follows that $\CZF(\mathbf{J}) \not \vdash \phi^\sigma$.
\end{proof}

We have the following conjecture.

\begin{conjecture}
    Intuitionistic set theory $\IZF$ has the de Jongh property with respect to every intermediate logic.
\end{conjecture}

\section{Conclusions}
\label{Section: Conclusions}

In this paper, we defined a class of Kripke models for intuitionistic set theory $\IZF$. We then used these models to prove a range of de Jongh properties for $\IZF$ and $\CZF$.

\begin{question}
    Can we obtain independence results for $\IZF$ with blended models?
\end{question}

\begin{question}
    Is it possible to vary the construction of blended models to provide proper models of $\CZF$ (i.e., models of $\CZF$ that are not also models of $\IZF$)?
\end{question}

\printbibliography

\end{document}